\documentclass{gtmon_a}
\pdfoutput=1
\usepackage[all]{xy} 
\usepackage{stmaryrd}

\proceedingstitle{Proceedings of the Nishida Fest (Kinosaki 2003)}
\conferencestart{28 July 2003}
\conferenceend{8 August 2003}
\conferencename{International Conference in Homotopy Theory}
\conferencelocation{Kinosaki, Japan}

\editor{Matthew Ando}
\givenname{Matthew}
\surname{Ando}

\editor{Norihiko Minami}
\givenname{Norihiko}
\surname{Minami}

\editor{Jack Morava}
\givenname{Jack}
\surname{Morava}

\editor{W Stephen Wilson}
\givenname{W Stephen}
\surname{Wilson}

\title[Homotopy groups of homotopy fixed point spectra associated to
$E_n$]{Homotopy groups of homotopy fixed point spectra\\associated to $E_n$}

\author{Ethan S Devinatz}
\givenname{Ethan S}
\surname{Devinatz}
\address{Department of Mathematics\\
University of Washington\\\newline
Seattle\\
Washington\\
USA}
\email{devinatz@math.washington.edu}
\urladdr{http://www.math.washington.edu/~devinatz/}

\volumenumber{10}
\issuenumber{}
\publicationyear{2007}
\papernumber{6}
\startpage{131}
\endpage{145}

\doi{}
\MR{}
\Zbl{}

\arxivreference{}


\received{13 September 2004}
\revised{5 July 2005}
\accepted{}
\published{29 January 2007}
\publishedonline{29 January 2007}
\proposed{}
\seconded{}
\corresponding{}
\version{}

\makeatletter
\def\cnewtheorem#1[#2]#3{\newtheorem{#1}{#3}[section]
\expandafter\let\csname c@#1\endcsname\c@thm}

\let\xysavmatrix\xymatrix
\def\xymatrix{\disablesubscriptcorrection\xysavmatrix}
\AtBeginDocument{\let\bar\wbar\let\tilde\wtilde\let\hat\what}

\def\co{\mskip 0.5mu\colon\thinspace}

\newtheorem{thm}{Theorem}[section]    
\cnewtheorem{lem}[thm]{Lemma}          
\cnewtheorem{prop}[thm]{Proposition}
\cnewtheorem{cor}[thm]{Corollary}
\theoremstyle{definition}
\cnewtheorem{defn}[thm]{Definition}    
\cnewtheorem{remark}[thm]{Remark}

\makeatother  

\newcommand{\F}{\mathbb{F}}
\newcommand{\gal}{\mathrm{Gal}}
\newcommand{\bps}{BP_\ast}
\newcommand{\en}{E_n}
\newcommand{\ena}{E_{n\ast}}
\newcommand{\gn}{G_n}
\newcommand{\sn}{S_n}
\newcommand{\fpn}{\F_{p^n}}
\newcommand{\fp}{\F_p}
\newcommand{\wotimes}{\,\widehat\otimes\,}
\newcommand{\map}{\operatorname{Map}}
\newcommand{\mapc}{{\mathrm{Map}}_c}
\newcommand{\ext}{\mathrm{Ext}}

\newcommand{\lef}{\vartriangleleft}
\newcommand{\hn}{H_n}
\newcommand{\G}{\Gamma}
\newcommand{\Gn}{\G_n}
\newcommand{\Hom}{\operatorname{Hom}}
\newcommand{\holim}{\operatorname{holim}}

\renewcommand{\mod}{\operatorname{mod}}

\numberwithin{equation}{section}

\begin{document}

\begin{htmlabstract}
We compute the mod(p) homotopy groups of the continuous homotopy
fixed point spectrum E<sup>hH<sub>2</sub></sup><sub>2</sub> for
p&gt;2, where E<sub>n</sub> is the Landweber exact spectrum whose
coefficient ring is the ring of functions on the Lubin&ndash;Tate
moduli space of lifts of the height n Honda formal group law
over <b>F</b><sub>p<sup>n</sup></sub>, and H<sub>n</sub> is the subgroup
W<b>F</b><sup>&times;</sup><sub>p<sup>n</sup></sub>&#8906;Gal(<b>F</b><sub>p<sup>n</sup></sub>/<b>F</b><sub>p</sub>)
of the extended Morava stabilizer group G<sub>n</sub>. We examine some
consequences of this related to Brown&ndash;Comenetz duality and to
finiteness properties of homotopy groups of K(n)<sub>*</sub>&ndash;local
spectra. We also indicate a plan for computing
&pi;<sub>*</sub>(E<sup>hH<sub>n</sub></sup><sub>n</sub>&#8743;V(n-2)),
where V(n-2) is an E<sub>n*</sub>&ndash;local Toda complex.
\end{htmlabstract}

\begin{webabstract}
We compute the $\mathrm{mod}(p)$ homotopy groups of the continuous
homotopy fixed point spectrum $E^{hH_2}_2$ for $p>2$, where $E_n$
is the Landweber exact spectrum whose coefficient ring is the
ring of functions on the Lubin--Tate moduli space of lifts of
the height $n$ Honda formal group law over $\mathbb{F}_{p^n}$,
and $H_n$ is the subgroup $W\mathbb{F}^{\times}_{p^n}\rtimes
\mathrm{Gal}(\mathbb{F}_{p^n}/\mathbb{F}_p)$ of the extended Morava
stabilizer group $G_n$. We examine some consequences of this related
to Brown--Comenetz duality and to finiteness properties of homotopy
groups of $K(n)_{\ast}$--local spectra. We also indicate a plan for
computing $\pi_{\ast}(E^{hH_n}_n\wedge V(n{-}2))$, where $V(n{-}2)$
is an $E_{n\ast}$--local Toda complex.
\end{webabstract}

\begin{abstract}
We compute the $\mod(p)$ homotopy groups of the continuous homotopy fixed
point spectrum $E^{hH_2}_2$ for $p>2$, where $E_n$ is the Landweber
exact spectrum whose coefficient ring is the ring of functions on the
Lubin--Tate moduli space of lifts of the height $n$ Honda formal group
law over $\F_{p^n}$, and $H_n$ is the subgroup $W\F^\times_{p^n}\rtimes
\gal(\F_{p^n}/\F_p)$ of the extended Morava stabilizer group $G_n$. We
examine some consequences of this related to Brown--Comenetz duality
and to finiteness properties of homotopy groups of $K(n)_\ast$--local
spectra. We also indicate a plan for computing $\pi_\ast(E^{hH_n}_n\wedge
V(n{-}2))$, where $V(n{-}2)$ is an $E_{n\ast}$--local Toda complex.
\end{abstract}

\maketitle

\section*{Introduction}
Let $E_n$ denote the Landweber exact spectrum with coefficient ring
\[
E_{n\ast}=W\F_{p^n}\llbracket u_1,\ldots, u_{n-1}\rrbracket[u,u^{-1}],
\]
where $W\F_{p^n}$ denotes the ring of Witt vectors with coefficients in the
field $\F_{p^n}$ of $p^n$ elements, and whose $BP_\ast$--algebra structure
map $r\co BP_\ast\to E_{n\ast}$ is given by
\[
r(v_i)=\left\{\begin{array}{c@{\qquad}l}
u_iu^{1-p^i} & i<n \\
u^{1-p^n} & i=n \\
0 & i>n\end{array}\right.
\]
where $v_i\in\bps$ is the $i^{\mathrm{th}}$ Hazewinkel generator. In
particular, each $u_i$ has degree $0$ and $u$ has degree $-2$. $\en$ is a
commutative ring spectrum, and Morava theory tells us that the group of ring
automorphisms of $\en$ is isomorphic to the profinite group
$\gn=\sn\rtimes\gal$, where $\sn$ denotes the group of (not necessarily
strict) isomorphisms of the height $n$ Honda formal group law over $\fpn$,
and $\gal$ is the Galois group of $\fpn/\fp$. A priori, $\gn$ acts on $\en$
only in the stable category, but Hopkins and Miller (later improved by
Goerss and Hopkins) proved that this can be made an honest action in an
appropriate point set category of spectra (see Goerss and Hopkins \cite{9} and Rezk \cite{14}).
``Continuous homotopy fixed point spectra'' may also be constructed \cite{7}:
if $G$ is a closed subgroup of $\gn$, the continuous homotopy $G$ fixed
point spectrum will be denoted by $\en^{hG}$; if $G$ is finite, this
spectrum agrees with the ordinary homotopy fixed point spectrum. Moreover,
$\en^{h\gn}\simeq L_{K(n)}S^0$, the $K(n)_\ast$--localization of $S^0$,
$\en^{hG}$ has the expected functorial properties, and there is a strongly
convergent ``continuous homotopy fixed point spectral sequence''
\[
H^\ast_c(G,\en^\ast X)\Rightarrow (\en^{hG})^\ast X
\]
for any spectrum $X$. ($H^\ast_c(G,\en^\ast X)$ denotes the continuous
cohomology of $G$ with coefficients in the profinite $G$--module $\en^\ast
X$.)

The hope of this paper is to make some headway towards the computation of
$\pi_\ast\en^{hG}$, for $G$ a closed subgroup of $\gn$. At first sight, this
program seems impossible: the formulas for the action of (most elements of)
$\gn$ on $\ena$ are extremely complicated (see Devinatz and Hopkins \cite{6}), making the
computation of $H^\ast_c(G,\ena)$ apparently inaccessible. However,
\[
H^\ast_c(\gn,N)=\ext^\ast_{\mapc(G,\ena)}(\ena,N),
\]
where $(\ena,\mapc(G,\ena))$ is the complete Hopf algebroid defined using
the action of $G$ on $\ena$ (see for example Devinatz \cite{5}). Since
$\mapc(G,\ena)$ is a quotient of 
\[
\mapc(\gn,\ena)=\ena\wotimes_{\bps}\bps
BP\wotimes_{\bps}\ena\equiv\ena^\wedge\en,
\]
one may try to use the Hopf algebroid structure maps
in $\bps BP$ together with several Bockstein spectral sequences to go from,
for example, $H^\ast_c(G,\ena/I_n)$ to $H^\ast_c(G,\ena)$. As usual, $I_n$
is the maximal ideal $(p,u_1,\ldots,u_{n-1})$ in $\ena$.

Let $H_n=W\fpn^\times\rtimes\gal\subset\gn$, where $W\fpn^\times$ is the
subgroup of $\sn$ consisting of the diagonal matrices (see \fullref{sec1}), and let
$M(p)$ denote the $\mod(p)$ Moore spectrum. We compute
$\pi_\ast\bigl(\smash{E^{hH_2}_2}\wedge M(p)\bigr)$ for all primes $p>2$
(\fullref{thm3.8}). Of course, $\pi_\ast L_{K(2)}M(p)$ is known (Shimomura \cite{15,16}) for
$p>2$, so it is unclear if our computation yields any new homotopy
information. Our computation is, however, much simpler and already indicates
the necessity of ``$p$--adic suspensions'' in the Gross--Hopkins work on
Brown--Comenetz duality (\fullref{rem3.9}). Moreover, we believe that computations
such as $\pi_\ast\bigl(\en^{hH_n}\wedge V(n{-}2)\bigl)$---recall that the
Toda complex $V(n{-}2)$ exists $\ena$--locally whenever $p$ is sufficiently
large compared to $n$---should be accessible to more skilled calculators.

Even when a complete calculation of $\pi_\ast\en^{hG}$ is unattainable,
partial information can lead to interesting consequences. For example, it is
a long-standing conjecture that $\pi_\ast L_{K(n)}S^0$ is a module of finite
type over the $p$--adic integers $\Z_p$. 
(This conjecture is known to be true for $n=1$, and, if $p\ge 3$, for $n=2$
Shimomura and Wang \cite{17}, Shimomura and Yabe \cite{18}. The reader may
also find Hovey and Strickland \cite[Theorem~15.1]{12}
interesting, where it is shown that $\pi_\ast L_{K(2)}S^0$ is \emph{not} of
finite type---at least when $p\ge 5$---if the grading is taken to be over
the Picard group of invertible spectra in the $K(2)_\ast$--local category.)
By a thick
subcategory argument---see Devinatz \cite{3} for a discussion of this in the
$\ena$--local category---if $\pi_\ast L_{K(n)}X$ is of finite type for some
$X$ in the $\ena$--local category, then $\pi_\ast L_{K(n)}Y$ is of finite
type for any finite $Y$ such that
\[
\{m\le n:K(m)_\ast Y\ne 0\} \subset \{m\le n:K(m)_\ast X\ne 0\}.
\]
This in turn only requires that we prove that $\pi_\ast (\en^{hG}\wedge X)$
is of finite type for some closed subgroup $G$ of $\gn$ for which there
exists a chain
\[
G=K_0\lef K_1\lef\cdots\lef K_t=\gn
\]
of closed subgroups. Indeed, assume inductively that
$\pi_\ast(\en^{hK_i}\wedge X)$ is of finite type. Then, since $K_{i+1}/K_i$
is a $p$--adic analytic profinite group (see Dixon, du Sautoy, Mann and Segal \cite[Theorem 9.6]{8}), we have
that $H^\ast_c\bigl(K_{i+1}/K_i,\pi_\ast(\smash{\en^{hK_i}}\wedge Y)\bigr)$ is also
of finite type. (This follows from the fact that any $p$--adic analytic
profinite group is of type $p-FP_\infty$ in the language of Symonds and Weigel \cite{20}.) But,
in an earlier paper \cite{4}, we constructed a strongly convergent spectral sequence
\[
H^\ast_c\bigl(K_{i+1}/K_i,\pi_\ast\bigl(\smash{\en^{hK_i}}\wedge X\bigl)\bigr) \Rightarrow
\pi_\ast(\en^{hK_{i+1}}\wedge X)
\]
and showed that its $E_\infty$ term has a horizontal vanishing line. This
implies that the group $\pi_\ast\bigl(\smash{\en^{hK_{i+1}}}\wedge X\bigr)$ is of finite type and
hence, by induction, so is $\pi_\ast
L_{K(n)}X={\pi_\ast\bigl(\smash{\en^{h\gn}}\wedge X\bigr)}$.

These considerations are unfortunately not applicable to $G=\hn$, since the
normalizer of $\hn$ in $\gn$ is $\hn$, and, moreover, the group
$\pi_\ast\bigl(\smash{E^{hH_2}_2}\wedge M(p)\bigr)$ is not even of finite type. Yet
it is, in some sense, ``almost'' of finite type (see \fullref{sec4}), although the
significance of this property is not clear.

The author was partially supported by a grant from the NSF.

\setcounter{section}{0}
\section{$H^\ast_c(\hn,\ena/I_n)$ and its Hopf algebroid description}
\label{sec1}

Recall that the group $\sn$ may be described in several ways. If $\Gn$
denotes the height $n$ Honda formal group law over $\fpn$, then $\sn$
consists of all formal power series of the form $\sum_{i\ge 0}^{\Gn}
b_ix^{p^i}$ with each $b_i\in\fpn$ and $b_i\ne 0$. The ring of endomorphisms
of $\Gn$ may also be described as the ring obtained by adjoining an
indeterminate $S$---which corresponds to the endomorphism $f(x)=x^p$---to $W\fpn$ along with the relations $S^n=p$ and $Sw=w^\sigma S$, where
$\sigma\co W\fpn\to W\fpn$ denotes the Frobenius automorphism. The
automorphism $\sum_{i\ge 0}^{\Gn}b_ix^{p^i}$ corresponds to the element
$\sum^{n-1}_{i=0}a_iS^i$ with
\[
a_i=\sum_{k\ge 0}e(b_{i+nk})p^k,
\]
where $e(b)$ is the multiplicative representative of $b$ in $W\fpn$. The
subgroup $W\fpn^\times$ of $\sn$ is then the group of automorphisms with
$a_i=0$ for all $i>0$. In terms of matrices, $\sn$ is the subgroup of
$GL_n(W\fpn)$ consisting of matrices of the form
\[
\left[\begin{array}{ccccc}
a_0 & pa_{n-1} & pa_{n-2} & \cdots & pa_1 \\
a^{\sigma^{-1}}_1 & a^{\sigma^{-1}}_0 & pa^{\sigma^{-1}}_{n-1} & \cdots &
pa^{\sigma^{-1}}_2 \\
\vdots & \vdots & \vdots & & \vdots\\
\vdots & \vdots & \vdots & & pa^{\sigma^{-(n-2)}}_{n-1} \\
a^{\sigma^{-(n-1)}}_{n-1} & a^{\sigma^{-(n-1)}}_{n-2} &
a^{\sigma^{-(n-1)}}_{n-3} & \cdots & a^{\sigma^{-(n-1)}}_0
\end{array}\right]\quad ,
\]
and $W\fpn^\times$ is the subgroup of diagonal matrices in $\sn$. 

Now let $\sn^0$ be the $p$--Sylow subgroup of $\sn$ consisting of strict
automorphisms of $\Gn$. There is a split extension
\[
\sn^0\to \sn\to\fpn^\times;
\]
the map $\sn\to\fpn^\times$ is given by
$\sum^{n-1}_{i=0}a_iS^i\mapsto\overline{a_0}$, and the splitting sends
$a\in\fpn$ to $e(a)\in W\fpn^\times\subset\sn$. This map also gives us a
splitting of the short exact sequence
\[
0\to W\fpn^0\to W\fpn^\times\to\fpn^\times\to 0,
\]
and hence an isomorphism $W\fpn^\times\to W\fpn^0\times\fpn^\times$. Since
the order  of $\fpn^\times$ is prime to $p$, it follows that
\[
H^\ast_c(W\fpn^\times, N)\mathop{\longrightarrow}\limits^{\approx}
H^\ast_c(W\fpn^0, N)^{\fpn^\times}
\]
whenever $N$ is a discrete $\Z_p\llbracket W\fpn^\times\rrbracket$--module. 
Now suppose, in addition, that $N$ is a $W\fpn$--module and $H_n$--module in
such a way that the $W\fpn^\times$ action is $W\fpn$--linear and that
$\sigma(cn)=c^\sigma\sigma(n)$ for all $c\in W\fpn$ and $n\in N$. It then
follows from Devinatz \cite[Lemma 5.4]{1} that $H^i\bigl(\gal, H^\ast_c(W\fpn^\times,
N)\bigr)=0$ for all $i>0$, and hence 
\[
H^\ast_c(\hn, N)\mathop{\longrightarrow}\limits^{\approx}
H^\ast_c(W\fpn^\times, N)^{\gal}.
\]
Now $\sn$ acts on $\ena/I_n=\fpn[u,u^{-1}]$ via $\fpn$--algebra
homomorphisms, and the action on $u$ is given by 
\begin{equation}\label{eqn1.1}
\Bigl(\mathop{\textstyle{\sum}}\limits^{n-1}_{i=0}a_iS^i\Bigr)(u)=\overline{a_0}u,
\end{equation}
where, once again, $\overline{a_0}$ is the $\mod(p)$ reduction of $a_0$.
From this it follows that
\[
H^\ast_c(W\fpn^\times, \fpn[u,u^{-1}])=\fpn[v_n,v^{-1}_n]\otimes_{\fpn}
H^\ast_c(W\fpn^0, \fpn).
\]
Moreover, since $\gal$ acts trivially on $v_n$,
\[
H^\ast_c(\hn, \fpn[u,u^{-1}]) = \fp[v_n,v^{-1}_n]\otimes H^\ast_c(W\fpn^0,
\fpn)^{\gal}.
\]
It is also easy to compute $H^\ast_c(W\fpn^0,\fpn)^{\gal}$. Let
$g_i\in\Hom_c(W\fpn^0, \fpn)$ be defined by
\begin{equation}\label{eqn1.2}
g_i\Bigl(1+\mathop{\textstyle{\sum}}\limits_{j\ge 1}e(c_j)p^j\Bigr)=c_1^{p^i}=c^{\sigma^i}_1,
\end{equation}
$0\le i\le n-1$. Since the Galois automorphisms
$id,\sigma,\ldots,\sigma^{n-1}$ are linearly independent over $\fpn$, so are
the $g_i$'s. Now, and for the rest of this section, assume that $p>2$. Then
$$\Z^n_p\approx W\fpn \stackrel{\approx}{\longrightarrow} W\fpn^0$$
via
the map sending $x\in W\fpn$ to $\exp(px)=1+\smash{\sum_{j\ge
1}\frac{p^jx^j}{j!}}\in W\fpn^0$, so that $H^\ast_c(W\fpn^0, \fpn)$ is the
exterior algebra over $\fpn$ on $n$ generators in $H^1_c(W\fpn^0, \fpn)$.
This implies that these generators may be taken to be $g_0, g_1, \ldots,
g_{n-1}$. Each $g_i$ is Galois invariant, so
\begin{equation}\label{eqn1.3}
H^\ast_c(\hn,\fpn[u,u^{-1}]) = \fp[v_n,v^{-1}_n]\otimes
E(g_0,g_1,\ldots,g_{n-1}).
\end{equation}
Next consider the complete Hopf algebroid $\bigl(\ena, \mapc\bigl(W\fpn^0,
\ena\bigr)\bigr)\equiv
 (\ena,\Sigma_n)$. We explicitly identify
$\Sigma_n/I_n\Sigma_n$ as a quotient of $\ena^\wedge\en/I_n\ena^\wedge\en$
and give cobar representatives for
\[
g_i\in H^{1,0}_c(W\fpn^0, \fpn[u,u^{-1}]) =
\ext^{1,0}_{\Sigma_n/I_n\Sigma_n} (\fpn[u,u^{-1}], \fpn[u,u^{-1}]).
\]
First recall that the maps $\eta_L, \eta_R\co\ena\to\mapc(\gn,\ena)$ are given
by $\eta_R(x)(s)=x$, $\eta_L(x)(s)=s^{-1}x$. Since $W\smash{\fpn^0}\subset \gn$ acts
trivially on $W\fpn\subset \ena$, it follows that
$\eta_R\big|_{W\fpn}=\eta_L\big|_{W\fpn}$ in $\Sigma_n$, so that $\Sigma_n$
is a Hopf algebra over $W\fpn$ and is a quotient of
\begin{eqnarray*}
W\fpn\otimes_{\Z_p}\mapc(\sn,\ena)^{\gal} & = &
W\fpn\otimes_{\Z_p}(\ena^{\gal}\wotimes_{\bps}\bps
BP\wotimes_{\bps}\ena^{\gal}) \\
& = & W\fpn\otimes_{\Z_p}(\en^{\gal})^\wedge_\ast\en^{\gal},
\end{eqnarray*}
where $\en^\gal$ is the Landweber exact spectrum with coefficient ring
$$\Z_p\llbracket u_1,\ldots, u_{n-1}\rrbracket [u,u^{-1}].$$
Now let $u\equiv\eta_R(u)$ and $w\equiv\eta_L(u)$ in
$(\en^\gal)^\wedge_\ast\en^\gal$. By \eqref{eqn1.1}, we have that $u=w$ in
$\Sigma_n/I_n\Sigma_n$. Moreover, the image of $t_j\in\bps BP$ in
$(\en^\gal)^\wedge_\ast\en^\gal$---also denoted $t_j$---satisfies
\[
t_j\Bigl(\sum_{i\ge 0}{}^{\Gn}b_ix^{p^i}\Bigr)=u^{1-p^i}b^{-1}_0b_j\quad\mod
I_n(\en^\gal)^\wedge_\ast\en^\gal
\]
(see Devinatz \cite[Proposition 2.11]{1}), and thus
\begin{equation}\label{eqn1.4}
\Sigma_n/I_n\Sigma_n=\fpn[u,u^{-1}][t_n,t_{2n},\ldots]/J_n,
\end{equation}
with
\[
J_n=(t^{p^n}_n-v^{p^n-1}_n t_n, t^{p^n}_{2n} - v^{p^{2n}-1}_n t_{2n},
\ldots, t^{p^n}_{jn} - v^{p^{jn}-1}_n t_{jn},\ldots).
\]
Finally, let $g=v^{-1}_nt_n\in\Sigma_n$. These considerations imply that
$g^{p^i}\in\Sigma_n/I_n\Sigma_n$ is a cobar representative for $g_i\in
H^1(W\fpn^0, \fpn)$.

\section{The Bockstein spectral sequence}
\label{sec2}

\setcounter{equation}{0}

Fix a prime $p$ and integer $n\ge 2$, and let $N$ be a complete
$\bigl(\ena,\map_c(\hn,\ena)\bigr)$--comodule. Write
$$H^\ast N\equiv H^\ast_c(\hn,N) =
H^\ast_c(W\fpn^0,N)^{\fpn^\times\rtimes\gal}
= \ext^\ast_{\Sigma_n}(\ena,N)^{\fpn^\times\rtimes\gal}.$$
The Bockstein spectral sequence we will use is defined by the exact couple
\begin{equation}\label{eqn2.1}
\xymatrix  @R=2pc @C=-1pc{
H^\ast(\ena/(p,u_1,\ldots,u_{n-2})) \ar[dr] &&
H^\ast(\ena/(p,u_1,\ldots,u_{n-2})) \ar[ll]_{v_{n-1}} \\
& H^\ast(\ena/I_n). \ar[ur]
}
\end{equation}
Truncated, this spectral sequence is isomorphic to the spectral sequence of
the unrolled exact couple
\[
\xymatrix @R=2pc @C=-1pc{
0 \ar[dr] && H^\ast(\ena/I_n) \ar[ll] \ar[dr] && H^\ast(\ena/(p,\ldots, u_{n-2},u^2_{n-1}))
\ar[ll] \ar[dr] && \cdots \ar[ll] \\
& H^\ast(\ena/I_n) \ar[ur] && H^\ast(\ena/I_n) \ar[ur]_{v_{n-1}} &&
H^\ast(\ena/I_n) \ar[ur]_{v^2_{n-1}}
}
\]
Since $H^{\ast\ast}(\ena/(p,u_1,\ldots, u_{n-2}, u^k_{n-1}))$ is finite
in each bidegree, this spectral sequence converges strongly to 
\[
H^\ast(\ena/(p,\ldots,u_{n-2}))=\lim_{\mathop{\leftarrow}\limits_{j}}
H^\ast(\ena/(p,u_1,\ldots, u_{n-2},u^j_{n-1})).
\]

\section{Computation of $\pi_\ast(E^{hH_2}\wedge M(p))$}
\label{sec3}

\setcounter{equation}{0}

In this section, we specialize the above spectral sequence to the case $n=2$,
$p>2$. Write $\wwbar{\Sigma}_2=\Sigma_2/p\Sigma_2$, and let
$g=v^{-1}_2t_2\in\wwbar{\Sigma}_2$ as in \fullref{sec1}.

We will need the following congruences for our calculation of the
differentials in the Bockstein spectral sequence.

\begin{lem}\label{lemma3.1}
$v^{1-p^2}_2t^{p^2}_2=t_2\quad\mod v^{p+1}_1\wwbar{\Sigma}_2$.
\end{lem}

\setcounter{equation}{1}
\begin{proof}
Begin with the formula (see Ravenel \cite[Theorem A2.2.5]{13})
\[
\sum_{i,j\ge 0}{}^F t_i\eta_R(v_j)^{p^i} = \sum_{i,j\ge 0}{}^F v_it^{p^i}_j
\]
in $\bps BP/p\bps BP$, where $F$ is the universal $p$--typical formal group
law on $\bps$. Up through power series degree $p^4$ we then have
\begin{multline}\label{eqn3.2}
v_1+_Fv^p_1t_1+_Fv^{p^2}_1t_2+_Fv^{p_3}_1t_3+_F\eta_R(v_2)+_Ft_1\eta_R(v_2)^p+_Ft_2\eta_R(v_2)^{p^2}
\\
= v_1+_Fv_1t^p_1+_Fv_1t^p_2+_Fv_1t^p_3+_Fv_2+_Fv_2t^{p^2}_1+_Fv_2t^{p^2}_2
\end{multline}
in $E^\wedge_{2\ast}E_2/pE^\wedge_{2\ast}E_2$. But
$\eta_R(v_2)=v_2+v_1t^p_1-v^p_1t_1$ in $\bps BP/p\bps BP$, and therefore,
since $t_1\in v_1\wwbar{\Sigma}_2$, $\eta_R(v_2)=v_2$ $\mod v^{p+1}_1\wwbar{\Sigma}_2$.
$t_3$ is also in $v_1\wwbar{\Sigma}_2$; hence, $\mod v^{p+1}_1\wwbar{\Sigma}_2$,
\eqref{eqn3.2} reduces to
\[
v^p_2t_1+_Fv^{p^2}_2t_2=v_1t^p_2+_Fv_2t^{p^2}_2.
\]
The desired result follows immediately from this equation.
\end{proof}

\setcounter{thm}{2}

\begin{lem}\label{lemma3.3}
$t_1=v_1g^p-v^{p+2}_1v^{-1}_2g+v^{p+2}_1v^{-1}_2g^p\quad\mod
v^{2p+3}_1\wwbar{\Sigma}_2$.
\end{lem}

\begin{proof}
From Ravenel \cite[Corollary 4.3.21]{13},
$$\eta_R(v_3)=
v_3{+}v_2t^{p^2}_1\!\!{+}v_1t^p_2{-}v^p_2t_1{-}v^{p^2}_1\!t_2
  {-}v^p_1t^{1{+}`p^2}_1
  \!\!{+} v^{p^2}_1\!t^{1{+}`p}_1{+}v_1w_1\bigl(v_2,v_1t^p_1,
`{-}v^p_1t_1\bigr)$$
in $\bps BP/p\bps BP$,
where $w_1(x,y,z)\equiv\frac{1}{p}[x^p+y^p+z^p-(x+y+z)^p]$. Hence
\setcounter{equation}{3}
\begin{equation}\label{eqn3.4}
0=v_1t^p_2-v^p_2t_1-v^2_1v^{p-1}_2t^p_1 + v^{p+1}_1v^{p-1}_2t_1\quad\mod
v^{2p+3}_1\wwbar{\Sigma}_2,
\end{equation}
and thus
\[
t_1=v^{-p}_2v_1t^p_2\quad\mod v^{p+2}_1\wwbar{\Sigma}_2.
\]
Plug this relation for $t_1$ back into the last two terms of \eqref{eqn3.4} to
get 
\[
v^p_2t_1=v_1t^p_2-v^{p+2}_1v^{-p^2+p-1}_2t^{p^2}_2+v^{p+2}_1v^{-1}_2t^p_2\quad\mod
v^{2p+3}_1\wwbar{\Sigma}_2.
\]
By the previous lemma,
\[
v^{p+2}_1v^{-p^2+p-1}_2t^{p^2}_2=v^{p+2}_1v^{p-2}_2t_2\quad\mod
v^{2p+3}_1\wwbar{\Sigma}_2.
\]
We then get the desired result.
\end{proof}

The next propositions will allow us to compute the Bockstein differentials
on $v^k_2\in H^0(\F_{p^2}[u,u^{-1}])$. 

\setcounter{thm}{4}
\begin{prop}\label{prop3.5}
In $\wwbar{\Sigma}_2/v^{3p+3}_1\wwbar{\Sigma}_2$, 
\begin{multline*}
\eta_R(v^s_2)-v^s_2=\\ sv^s_2\bigl[v^{-1}_2 v^{1+p}_1(g^{p^2}-g^p)+v^{-2}_2
v^{2(1+p)}_1 (g-g^p)+\tfrac{s-1}{2}v^{-2}_2 v^{2(1+p)}_1
(g^{p^2}-g^p)^2\bigr].
\end{multline*}
\end{prop}

\begin{proof}
Compute in $\wwbar{\Sigma}_2/v^{3p+3}_1\wwbar{\Sigma}_2$: 
\begin{eqnarray*} 
\eta_R(v^s_2) - v^s_2 & = & v^s_2[(v^{-1}_2\eta_R(v_2))^s-1] \\
& = & v^s_2[(1+v^{-1}_2v_1t^p_1-v^{-1}_2v^p_1t_1)^s-1] \\
& = &
sv^s_2\bigl[v^{-1}_2(v_1t^p_1-v^p_1t_1)+\tfrac{s-1}{2}v^{-2}_2(v_1t^p_1-v^p_1t_1)^2\bigr]
\\
& = &
sv^s_2\bigl[v^{-1}_2v^{p+1}_1g^{p^2}-v^{-1}_2v^{p+1}_1g^p+v^{-2}_2v^{2(p+1)}_1
(g-g^p)\\
&&\qquad + \tfrac{s-1}{2}v^{-2}_2v^{2(p+1)}_1 (g^{p^2}-g^p)^2\bigr]
\end{eqnarray*}
by the previous lemma.
\end{proof}

\begin{prop}\label{prop3.6}
Suppose that $p\not|\;s$, that $k\ge 0$, and let
$$b(sp^k)=sp^k-p^k-p^{k-1}-\cdots-1.$$
Then there exists $z_{sp^k}\in
E_{2^\ast}$ such that $z_{sp^k}=v^{sp^k}_2\mod I_2$ and such that
\[
dz_{sp^k}\equiv \eta_R(z_{sp^k})-z_{sp^k}=cv^{b(sp^k)}_2
v^{(p^k+p^{k-1}+\cdots+1)(p+1)}_1 (g^p-g)
\]
in $\wwbar{\Sigma}_2/v^{(p^k+p^{k-1}+\cdots+p+2)(p+1)}_1
\wwbar{\Sigma}_2$, for some $c\in\F^\times_p$.
\end{prop}

\begin{proof}
Proceed by induction on $k$. If $k=0$, let $z_s=v^s_2$. Then by the preceding
proposition,
\[
\eta_R(z_s)-z_s=sv^{s-1}_2v^{p+1}_1(g^{p^2}-g^p)\mod
v^{2(p+1)}_1\wwbar{\Sigma}_2.
\]
But $g^{p^2}=g\mod v^{p+1}_1\wwbar{\Sigma}_2$, so we get the desired result.

Suppose now that $z_{sp^{k-1}}$ has been chosen. Then 
\begin{multline*}
d(z_{sp^{k-1}})^p=cv_2^{pb(sp^{k-1})}v^{(p^k+p^{k-1}+\cdots+p)(p+1)}_1
(g^{p^2}-g^p)\\
\bigl(\mod v^{(p^k+p^{k-1}+\cdots+p^2+2p)(p+1)}_1\wwbar{\Sigma}_2\bigr).
\end{multline*}
Next consider $dv^{(s-1)p^k-p^{k-1}-\cdots-p+1}_2$. Since the exponent of
$v_2$ is equal to $1\mod(p)$, we have
\[
d(v^{b(sp^k)+2}_2)=v^{b(sp^k)+1}_2v^{1+p}_1(g^{p^2}-g^p) + v^{b(sp^k)}_2
v^{2(1+p)}_1(g-g^p)\qua\mod v^{3p+3}_1\wwbar{\Sigma}_2.
\]
Then take
\[
z_{sp^k}=(z_{sp^{k-1}})^p-cv_1^{(p^k+p^{k-1}+\cdots+p-1)(p+1)}v_2^{(s-1)p^k
- p^{k-1}-\cdots-p+1}.\proved
\]
\end{proof}

\begin{cor}\label{cor3.7}
Suppose that $p\not|s$ and $k\ge 0$. Up to multiplication by a unit in
$\fp$,
\[
d_{(p^k+p^{k-1}+\cdots+1)(p+1)} v_2^{sp^k} =
v_2^{(s-1)p^k-p^{k-1}-\cdots-p-1}(g^p-g)
\]
in the Bockstein spectral sequence \eqref{eqn2.1}. In particular, $v^t_2(g^p-g)$ is a
boundary for all $t\in\Z$.
\end{cor}

Now, since $g^{p^i}$ may be regarded as an element of $H^1_c(H_n,\fpn)$, and
the inclusion
$$\fpn\longrightarrow(E_n)_0/(p,u_1,\ldots, u_{n-2})$$
induces a map
$$H^\ast_c(H_n, \fpn)\longrightarrow H^\ast(\ena/(p_1,u_1,\ldots, u_{n-2})),$$
it follows
that $g^{p^i}$ is a permanent cycle in the Bockstein spectral sequence.
Moreover,
\[
H^\ast(\F_{p^2}[u,u^{-1}]) = \fp[v_2,v^{-1}_2] \otimes E(g+g^p,g-g^p).
\]
Using the preceding corollary, it's easy to read off the remaining
differentials, yielding our main result.

\begin{thm}\label{thm3.8}
If $p\ge 3$
\[
H^i(\F_{p^2}\llbracket u_1\rrbracket [u,u^{-1}]) =
\left\{\begin{array}{l@{\hspace{.5in}}l}
\fp[v_1]\{1\} & i=0 \\
\fp[v_1]\{\zeta\}\times\prod_{t\in\Z} \frac{\fp[v_1]}{(v^{n_t}_1)}\{c_t\} & i=1
\\
\prod_{t\in\Z} \frac{\fp[v_1]}{(v_1^{n_t})}\{c_t\zeta\} & i=2
\end{array}\right.
\]
as $\fp[v_1]$--modules, where $\zeta=g+g^p$, $c_t$ reduces to
$v^t_2(g-g^p)\in H^1(\F_{p^2}[u,u^{-1}])$, and 
\[
n_t=\left\{\begin{array}{l@{\quad}l}
p+1 & t\ne -1\quad\mod(p) \\
(p^i+p^{i-1}+\cdots+1)(p+1) & t=(s-1)p^i-p^{i-1}-\cdots-p-1, \\
&s\ne 0\quad\mod(p)
\end{array}\right.
\]
By sparseness,
\[
\pi_{t-s}(E^{hH_2}_2\wedge M(p)) \approx
H^{s,t}(\F_{p^2}\llbracket u_1\rrbracket [u,u^{-1}]).
\]
\end{thm}

\begin{remark}\label{rem3.9}
Let $I_n$ denote the Brown--Comenetz dual of $L_nS^0$, the
$\ena$--localization of $S^0$. $I_n$ is characterized by
\[
\pi_0F(X,I_n)=[X,I_n]_0=\Hom(\pi_0L_nX,\Q/\Z_{(p)})
\]
for any spectrum $X$. In \cite{10} (see also Strickland \cite{19}), Gross and Hopkins
establish a remarkable relationship between Brown--Comenetz and
Spanier--Whitehead duality: they prove that if $p$ is sufficiently large
compared to $n\ge 2$, and if $X$ is a $K(n{-}1)_\ast$--acyclic finite complex
with $p\ena X=0$ and with $v_n$ self-map $\Sigma^{2p^N(p^n-1)}X\to X$, then 
\setcounter{equation}{9}
\begin{equation}\label{eqn3.10}
F(X,I_n)\simeq\Sigma^\alpha L_nDX,
\end{equation}
where $\alpha$ is any integer with
\[
\alpha=2p^{nN}(p^n-1/p-1) + n^2-n\quad\mod(2p^N(p^n-1)).
\]
(As usual, $DX$ denotes the Spanier--Whitehead dual of $X$.) There is,
however, no integer $\alpha$ for which \eqref{eqn3.10} is satisfied for
\emph{all} $X$. This contrasts with the situation when $n=1$: here we have
$I_1\simeq \Sigma^2L_1(S^0_p)$ (if $p>2$), where $S^0_p$ denotes $S^0$
completed at $p$, and thus $F(X,I_1)\simeq\Sigma^2L_1DX$ whenever $X$ is a
rationally acyclic finite spectrum.

Historically, it was Shimomura's calculation \cite{15} of $\pi_\ast L_2M$
which shattered the hope that $I_2$ might also be an integral suspension of
$L_2(S^0_p)$. Our calculation of $\pi_\ast(E^{hH_2}_2\wedge M(p))$ yields
this result as well; a sketch of the proof follows.

Suppose there existed an integer $c$ with 
\begin{equation}\label{eqn3.11}
F(M(p,v^k_1),I_2) \simeq \Sigma^cL_2DM(p,v^k_1)
\end{equation}
for a cofinal set of $k$, where $M(p,v^k_1)$ denotes a finite spectrum with
\[
\bps M(p,v^k_1)=\bps/(p,v^k_1).
\]
In addition, we may assume that
\[
DM(p,v^k_1)\simeq\Sigma^{-2k(p-1)-2}M(p,v^k_1).
\]
Let
\[
E_{2\ast}M(p,v^k_1)^{\hbox{\footnotesize$\sim$}} = \Hom(E_{2\ast}M(p,v^k_1),\Q/\Z_{(p)}),
\]
and recall that
\[
\Sigma^4(E_{2\ast}M(p,v^k_1)^{\hbox{\footnotesize$\sim$}})\approx E_{2\ast}F(M(p,v^k_1),I_2)
\]
as modules over $E_{2\ast}$ and $G_2$. (See Strickland \cite[Proposition 17]{19} or
Devinatz \cite{2} for $p\ge 5$; note, however, that we are using Strickland's
equivalent
definition of $E_{2\ast}M(p,v^k_1)^{\hbox{\footnotesize$\sim$}}$.) Then
\eqref{eqn3.11} implies that
\[
E_{2\ast}M(p,v^k_1)^{\hbox{\footnotesize$\sim$}}\approx\Sigma^{c-2k(p-1)-6}E_{2\ast}M(p,v^k_1),
\]
and, by the theory of Poincar\'e pro--$p$ groups (cf.\ Devinatz and Hopkins \cite[Sections 5,
6]{6}), there is a map
\[
H^{2,6+2k(p-1)-c}(E_{2\ast}M(p,v^k_1))\to\Q/\Z_{(p)}
\]
such that
\[
H^i(E_{2\ast}M(p,v^k_1)) \otimes H^{2-i}(E_{2\ast}M(p,v^k_1))\to
H^2(E_{2\ast}M(p,v^k_1))\to\Q/\Z_{(p)}
\]
is a perfect pairing. Hence there must exist
a $d_k
\in H^{2,6+2(p-1)-c}(E_{2\ast}M(p,v^k_1))$, for each $k$, such that $v^{k-1}_1d_k\ne 0$.
But the computation of $H^2(\F_{p^2}\llbracket u_1\rrbracket[u,u^{-1}])$ together
with the exact sequence
\[
H^2(\F_{p^2}\llbracket u_1\rrbracket [u,u^{-1}])\mathop{\longrightarrow}\limits^{v^k_1}
H^2(\F_{p^2}\llbracket u_1\rrbracket [u,u^{-1}]) \to H^2(E_{2\ast}M(p,v^k_1))\to 0
\]
shows that this is impossible.
\end{remark}

\section{Some remarks on finiteness}
\label{sec4}

\setcounter{equation}{0}

In this section, we work in the $\ena$--local stable category, so that by a
finite spectrum, we mean an object of the thick subcategory generated by
$L_nS^0$.

Let $G$ be a closed subgroup of $\gn$, $n\ge 1$.

\begin{prop}\label{prop4.1}
Let $Y$ be a $K(n{-}1)_\ast$--acyclic finite spectrum. Then
$\pi_\ast(E^{hG}_n\wedge Y)$ is of finite type (as a graded abelian group).
\end{prop}

\begin{proof}
The proof is just as we argued in the Introduction: use the strongly
convergent spectral sequence
\[
H^{\ast\ast}_c(G,\ena Y)\Rightarrow \pi_\ast(E^{hG}_n\wedge Y)
\]
whose $E_\infty$ term has a horizontal vanishing line. Since $\ena Y$ is of
finite type, so is $H^{\ast\ast}_c(G,\ena Y)$. The horizontal vanishing line
then implies that $\pi_\ast(E^{hG}_n\wedge Y)$ is also of finite type.
\end{proof}

Now suppose $n\ge 2$ and $X$ is a $K(n{-}2)_\ast$--acyclic finite spectrum
with $v_{n-1}$ self-map $\nu$.  Let $X(\nu^k)$ denote the cofiber of
$\nu^k\co\Sigma^{k|\nu|}X\to X$, and let $X(\nu^\infty)$ denote the cofiber of
$X\to\nu^{-1}X$, so that
$$X(\nu^\infty)=\holim\displaylimits_{\rightarrow k}
  \Sigma^{-k|\nu|}X(\nu^k).$$
There are also canonical maps
$X(\nu^k)\to X(\nu^{k-1})$
and
$X\to\holim\displaylimits_{\leftarrow k}X(\nu^k)$.
We will need the following well-known result (cf. Hovey \cite[Section 2]{11}).

\begin{lem}\label{lemma4.2}
If $Z$ is any ($\ena$--local) spectrum, the map
$$Z\wedge X\to\holim\displaylimits_{\leftarrow k}Z\wedge X(\nu^k)$$
is the $K(n)_\ast$--localization of $Z\wedge X$.
\end{lem}

\begin{prop}\label{prop4.3}
$\nu^{-1}\pi_\ast(\en^{hG}\wedge X)$ is countable if and only if
$\pi_\ast(\en^{hG}\wedge X)$ is of finite type.
\end{prop}

\begin{proof}
\fullref{prop4.1} implies that $\pi_\ast(\en^{hG}\wedge
X(\nu^\infty))$ is countable, and therefore $\nu^{-1}\pi_\ast(\en^{hG}\wedge
X)$ is countable if and only if $\pi_\ast(\en^{hG}\wedge X)$ is countable.
But
$$\en^{hG}\wedge X\simeq\holim\displaylimits_{\leftarrow k} \en^{hG}\wedge
X(\nu^k);$$
it therefore again follows from \fullref{prop4.1} that
$\pi_i(\en^{hG}\wedge X)$ is profinite and is thus countable if and only if
it's finite.
\end{proof}

\begin{remark}\label{rem4.4}
The chromatic splitting conjecture (see Hovey \cite{11}) actually identifies
$\nu^{-1}(\en^{h\gn}\wedge X)=\nu^{-1}L_{K(n)}X$ as
$L_{n-1}X\lor\Sigma^{-1}L_{n-1}X$.
\end{remark}

Although $\pi_\ast(E^{hH_2}_2\wedge M(p))$ is not of finite type, this
proposition suggests to us the sense in which it is ``almost'' of finite
type. The details follow.

We will consider graded modules over the graded ring $\F_p[\nu]$, where
$\nu$ has positive even (unless $p=2$) degree, satisfying the following two
conditions:
\begin{enumerate}
\item[(i)]$M$ is complete in the sense that $M=\mathop{\lim}\limits_{\leftarrow i}M/\nu^iM$.
\item[(ii)]$M/\nu M$ is an $\fp$ vector space of finite type.
\end{enumerate}

\begin{prop}\label{prop4.5}
Let $X$ and $\nu$ be as above and suppose that $p\co X\to X$ is trivial. Then
$\pi_\ast(\en^{hG}\wedge X)$ is an $\fp[\nu]$--module satisfying
conditions (i)
and (ii).
\end{prop}

\begin{proof}
Since
\[
\frac{\pi_\ast(\en^{hG}\wedge X)}{\nu^k\pi_\ast(\en^{hG}\wedge X)}
\hookrightarrow \pi_\ast(\en^{hG}\wedge X(\nu^k)),
\]
we have the requisite finiteness. Moreover, it follows from the commutative
diagram
\[
\xymatrix {
\pi_\ast(\en^{hG}\wedge X) \ar[r]^<<<<<<{\approx} \ar[d] &
\lower20pt\hbox{$\mathop{\lim}\limits_{\mathop{\leftarrow}\limits_{k}}
\pi_\ast(\en^{hG}\wedge X(\nu^k))$} \\
\mathop{\lim}\limits_{\mathop{\leftarrow}\limits_{k}}\frac{\pi_\ast(\en^{hG}\wedge
X)}{\nu^k\pi_\ast(\en^{hG}\wedge X)} \ar@{^{(}->}[ur]
}
\]
that $\pi_\ast(\en^{hG}\wedge X)$ is complete.
\end{proof}

Torii \cite[Proposition 4.10]{21} shows that such a module $M$ may be
written as
\setcounter{equation}{5}
\begin{equation}\label{eqn4.6}
M\approx\prod_\alpha \Sigma^{n_\alpha}\fp[\nu]\times\prod_\beta
\Sigma^{m_\beta}\fp[\nu]/(\nu^{i_\beta}).
\end{equation}
If $M$ is of finite type, then the $n_\alpha$'s are bounded below and
\[
\nu^{-1}M\approx \nu^{-1}\prod_\alpha\Sigma^{n_\alpha}\fp[\nu] =
\mathop{\bigoplus}\limits_{\alpha} \Sigma^{n_\alpha} \fp[\nu,\nu^{-1}].
\]
In general, the torsion submodule $T$ of $M$ is a submodule of
$\prod_{\beta}\Sigma^{m_\beta}\fp[\nu]/(\nu^{i_\beta})$; its closure $\wwbar
T$ is equal to $\prod_\beta\Sigma^{m_\beta}\fp[\nu]/(\nu^{i_\beta})$. Let us
say that $M$ is \emph{essentially of finite rank} if there are only a finite
number of $\alpha$ in the decomposition \eqref{eqn4.6}; that is, if and only if
$M/\wwbar T$ is a finitely generated $\fp[\nu]$--module.

Our main theorem shows that $\pi_\ast(E^{hH_2}_2\wedge M(p))$ is essentially
of finite rank for $p>2$. We do not know, however, whether this property is
generic; that is, whether, given $G$, $\pi_\ast(\en^{hG}\wedge X)$ is
essentially of finite rank for all $X$ satisfying the hypotheses of
\fullref{prop4.5} if it is for one such $X$ with $K(n{-}1)_\ast X\ne
0$.

\bibliographystyle{gtart}
\bibliography{link}

\begin{thebibliography}{}
\providecommand\bibmarginpar{\leavevmode\marginpar}
\def\urlstyle#1{{\tt #1}}

\bibitem{1}
\textbf{E\,S Devinatz}, \emph{Morava's change of rings theorem}, from: ``The
  \v{C}ech centennial (Boston, 1993)'', Contemp. Math. 181, Amer. Math. Soc.,
  Providence, RI (1995)  83--118 \xox{MR}{1320989}

\bibitem{2}
\textbf{E\,S Devinatz},
  \href{http://muse.jhu.edu/journals/american_journal_of_mathematics/v119/119.%
4devinatz.pdf} {\emph{Morava modules and {B}rown--{C}omenetz duality}}, Amer.
  J. Math. 119 (1997) 741--770 \xox{MR}{1465068}

\bibitem{3}
\textbf{E\,S Devinatz}, \emph{The generating hypothesis revisited}, from:
  ``Stable and unstable homotopy (Toronto, ON, 1996)'', Fields Inst. Commun.
  19, Amer. Math. Soc., Providence, RI (1998)  73--92 \xox{MR}{1622339}

\bibitem{4}
\textbf{E\,S Devinatz}, \href{http://dx.doi.org/10.1090/S0002-9947-04-03394-X}
  {\emph{A Lyndon--Hochschild--Serre spectral sequence for certain homotopy
  fixed point spectra}}, Trans. Amer. Math. Soc. 357 (2005) 129--150
  \xox{MR}{2098089}

\bibitem{5}
\textbf{E\,S Devinatz}, \href{http://dx.doi.org/10.1016/j.jpaa.2005.01.006}
  {\emph{Recognizing {H}opf algebroids defined by a group action}}, J. Pure
  Appl. Algebra 200 (2005) 281--292 \xox{MR}{2147271}

\bibitem{6}
\textbf{E\,S Devinatz}, \textbf{M\,J Hopkins},
  \href{http://links.jstor.org/sici?sici=0002-9327(199506)117:3%3C669:TAOTMS%3%
E2.0.CO%3B2-M} {\emph{The action of the {M}orava stabilizer group on the
  {L}ubin--{T}ate moduli space of lifts}}, Amer. J. Math. 117 (1995) 669--710
  \xox{MR}{1333942}

\bibitem{7}
\textbf{E\,S Devinatz}, \textbf{M\,J Hopkins},
  \href{http://dx.doi.org/10.1016/S0040-9383(03)00029-6} {\emph{Homotopy fixed
  point spectra for closed subgroups of the {M}orava stabilizer groups}},
  Topology 43 (2004) 1--47 \xox{MR}{2030586}

\bibitem{8}
\textbf{J\,D Dixon}, \textbf{M\,P\,F du~Sautoy}, \textbf{A Mann}, \textbf{D
  Segal}, \emph{Analytic pro--{$p$} groups}, Cambridge Studies in Advanced
  Mathematics 61, Cambridge University Press, Cambridge (1999)
  \xox{MR}{1720368}

\bibitem{9}
\textbf{P\,G Goerss}, \textbf{M\,J Hopkins}, \emph{Moduli spaces of commutative
  ring spectra}, from: ``Structured ring spectra'', London Math. Soc. Lecture
  Notes 315, Cambridge Univ. Press, Cambridge (2004)  151--200
  \xox{MR}{2125040}

\bibitem{10}
\textbf{M\,J Hopkins}, \textbf{B\,H Gross}, \emph{The rigid analytic period
  mapping, {L}ubin--{T}ate space, and stable homotopy theory}, Bull. Amer.
  Math. Soc. $($N.S.$)$ 30 (1994) 76--86 \xox{MR}{1217353}

\bibitem{11}
\textbf{M Hovey}, \emph{Bousfield localization functors and {H}opkins'
  chromatic splitting conjecture}, from: ``The \v{C}ech centennial (Boston,
  1993)'', Contemp. Math. 181, Amer. Math. Soc., Providence, RI (1995)
  225--250 \xox{MR}{1320994}

\bibitem{12}
\textbf{M Hovey}, \textbf{N\,P Strickland}, \emph{Morava {$K$}--theories and
  localisation}, Mem. Amer. Math. Soc. 139 (1999) \xox{MR}{1601906}

\bibitem{13}
\textbf{D\,C Ravenel},
  \href{http://www.math.rochester.edu/people/faculty/doug/mu.html}
  {\emph{Complex cobordism and stable homotopy groups of spheres}}, second
  edition, Amer. Math. Soc., Providence, RI (2004)

\bibitem{14}
\textbf{C Rezk}, \emph{Notes on the {H}opkins--{M}iller theorem}, from:
  ``Homotopy theory via algebraic geometry and group representations (Evanston,
  IL, 1997)'', Contemp. Math. 220, Amer. Math. Soc., Providence, RI (1998)
  313--366 \xox{MR}{1642902}

\bibitem{15}
\textbf{K Shimomura}, \emph{On the {A}dams--{N}ovikov spectral sequence and
  products of {$\beta$}--elements}, Hiroshima Math. J. 16 (1986) 209--224
  \xox{MR}{837322}

\bibitem{16}
\textbf{K Shimomura}, \emph{The homotopy groups of the $L_2$--localized mod
  {$3$} {M}oore spectrum}, J. Math. Soc. Japan 52 (2000) 65--90
  \xox{MR}{1727130}

\bibitem{17}
\textbf{K Shimomura}, \textbf{X Wang},
  \href{http://dx.doi.org/10.1016/S0040-9383(01)00033-7} {\emph{The homotopy
  groups $\pi_{*}(L_{2}S^{0})$ at the prime 3}}, Topology 41 (2002) 1183--1198
  \xox{MR}{1923218}

\bibitem{18}
\textbf{K Shimomura}, \textbf{A Yabe},
  \href{http://dx.doi.org/10.1016/0040-9383(94)00032-G} {\emph{The homotopy
  groups $\pi_{*}(L_{2}S^{0})$}}, Topology 34 (1995) 261--289 \xox{MR}{1318877}

\bibitem{19}
\textbf{N\,P Strickland},
  \href{http://dx.doi.org/10.1016/S0040-9383(99)00049-X}
  {\emph{Gross--{H}opkins duality}}, Topology 39 (2000) 1021--1033
  \xox{MR}{1763961}

\bibitem{20}
\textbf{P Symonds}, \textbf{T Weigel}, \emph{Cohomology of {$p$}--adic analytic
  groups}, from: ``New horizons in pro--$p$ groups'', Progr. Math. 184,
  Birkh\"auser, Boston (2000)  349--410 \xox{MR}{1765127}

\bibitem{21}
\textbf{T Torii},
  \href{http://muse.jhu.edu/journals/american_journal_of_mathematics/v125/125.%
5torii.pdf} {\emph{On degeneration of one-dimensional formal group laws and
  applications to stable homotopy theory}}, Amer. J. Math. 125 (2003)
  1037--1077 \xox{MR}{2004428}

\end{thebibliography}

\end{document}